\documentclass{amsart}
\usepackage{amsmath}
\usepackage{amssymb}
\usepackage{amsfonts}
\usepackage{url}
\usepackage{textcomp}
\usepackage[dvipdfmx]{graphicx}
\usepackage{setspace}
\usepackage{geometry} 
\usepackage{braket} 
\makeatletter
\@namedef{subjclassname@2020}{%
  \textup{2020} Mathematics Subject Classification}
\makeatother

\newtheorem{theorem}{Theorem}
\newtheorem{corollary}[theorem]{Corollary}
\newtheorem{proposition}[theorem]{Proposition}
\newtheorem{lemma}[theorem]{Lemma}
\theoremstyle{definition}
\newtheorem{example}[theorem]{Example}

\newcommand{\F}{\mathbb{F}}

\title{Generalised Howe curves of genus five attaining the Serre bound}

\author[M. Q. Kawakita]{Motoko Qiu Kawakita}
\address{Division of Mathematics, Shiga University of Medical Science,
 Seta Tsukinowa-cho, Otsu, Shiga, 520-2192 Japan}

\email{kawakita@belle.shiga-med.ac.jp}

\thanks{This research was partially supported by
JSPS Grant-in-Aid for Scientific Research (C) 23K03199.}

\begin{document}
\maketitle
\begin{abstract} 
We find that non-hyperelliptic generalised Howe curves and their twists
of genus $5$ attain the Hasse--Weil--Serre bound over some finite fields of order
$p$, $p^2$ or $p^3$ for a prime $p$. 
We are able to decompose their Jacobians completely under certain assumptions
and to determine the precise condition on the ﬁnite ﬁeld over which they attain the Hasse--Weil--Serre bound.
\end{abstract}

\section{Introduction}
Let $p$ be a prime, $k$ be a field of characteristic $p$
and $\F_{q}$ be a finite field with $q$ elements where $q$ is a power of $p$.
A curve $C$ is a projective, absolutely irreducible, non-singular algebraic curve defined over $k$. 
A curve $C$ over $\F_q$ 
is said to be {\em maximal} if the number of its rational points attains the Hasse--Weil upper bound
\begin{align*}
\#C(\F_q) \le q + 1 + 2g\sqrt{q} 
\end{align*}
where $g$ is the genus of $C$. 
In 1983, Serre provided a non-trivial improvement of the Hasse--Weil bound 
when $q$ is not a square root in \cite{s83}, namely
\begin{align*}
\#C(\F_q) \le q + 1 + g \lfloor 2\sqrt{q} \rfloor 
\end{align*}
where $\lfloor\cdot\rfloor$ is the floor function. 
We refer to this bound as the {\em Serre bound}. 

Curves attaining the Hasse--Weil or the Serre bound are interesting objects 
not only in their own right but also for their applications in coding theory. Indeed, 
Goppa described a way to use algebraic curves to construct linear error-correcting codes in \cite{goppa},
the so-called algebraic geometric codes; see \cite{sti}. 
The existence of curves with many rational points with respect to their genus
guarantee efficient error-correcting codes.
For this reason, maximal curves and curves attaining the Serre bound have been widely investigated in the last years,
see for instance \cite{ggs10, gt, gmz15}.

Howe investigated the non-singular projective model of the fibre product of two elliptic curves in \cite{how16},
which was called a Howe curve in subsequent works \cite{khh20, khs20}.
After that, he constructed curves of genus $5$, $6$ and $7$ 
by taking fibre products of curves of genus $1$ or $2$ in \cite{how17}.
He also implemented his constructions on a computer
and find a lot of curves of genus $4$, $5$, $6$ and $7$ with many rational points which updated the table \cite{ghlr}.
On the other hand, Richelot isogenies of Jacobians of superspecial 
curves of genus $1$ and $2$ are used in post-quantum cryptography.
Studying on decomposed Richelot isogenies of Jacobians of curves of higher genus 
is important for cryptography; see \cite{cs20, kt20, kt23} and the references there.
For this reason, Katsura and Takashima defined a generalised Howe curve in \cite{kt23}
by a natural extension of Howe's constructions in \cite{how16, how17}.
This formulation is very useful for searching curves with many rational points,
and actually stimulates this paper.
It leads us to discover new curves attaining the Serre bound. 

The paper is organised as follows.
Sections \ref{gh} and \ref{twistedE} are preparations for Section \ref{genus5}.
Section~\ref{gh} introduces and extends the notion of a generalised Howe curve from \cite{kt23}.
Section \ref{twistedE} provides the conditions for certain twisted Legendre elliptic curves 
to attain the Serre bound over $\F_p$, $\F_{p^2}$ or $\F_{p^3}$. 
Our main result is Theorem \ref{g5serre} in Section \ref{genus5},
which provides conditions for some twisted generalised Howe curves of genus $5$.
We implemented by Magma \cite{magma} and obtained explicit examples satisfying the conditions.
Obviously, they are new curves, and listed in Section~\ref{genus5}.

\section{A Generalised Howe curve} \label{gh}
We introduce the definition of a generalised Howe curve from \cite{kt23}.
Originally, it was defined over an algebraically closed field of characteristic $p>2$.
In this paper we extend it to a field $k$ of characteristic $p>2$. 
Let $C_1$, $C_2$ be the non-singular projective models of 
two hyperelliptic curves defined by
\begin{align*}
C_1 &\colon y_1^2=(x-a_1)(x-a_2)\ldots (x-a_r)(x-a_{r+1})\ldots (x-a_{2g_1+2}),\\
C_2 &\colon y_2^2=(x-a_1)(x-a_2)\ldots (x-a_r)(x-b_{r+1})\ldots (x-b_{2g_2+2})
\end{align*}
respectively, where $a_i, b_i \in k$, $a_i \neq b_j$ for any $i$ and $j$,
$a_i\neq a_j$ and $b_i \neq b_j$ for $i \neq j$, 
$0< g_1 \leq g_2$. 
Let $\psi_1 \colon C_1\longrightarrow {\bf P}^1$ and  $\psi_2 \colon C_2 \longrightarrow {\bf P}^1$ 
be the hyperelliptic structures. There are $r$ common branch points 
of $\psi_1$ and $\psi_2$.
Consider the fiber product $C_1 \times_{{\bf P}^1} C_2$:
\begin{align*}
\begin{array}{ccc}
   C_1 \times_{{\bf P}^1} C_2 & \stackrel{\pi_2}{\longrightarrow} & C_2 \\
     \pi_1\downarrow & & \downarrow  \psi_2\\
      C_1 & \stackrel{\psi_1}{\longrightarrow} & {\bf P}^1.
\end{array}
\end{align*}
Assume that there exists no isomorphism $\varphi\colon C_1 \longrightarrow C_2$
such that $\psi_2\circ \varphi = \psi_1$. Then, the curve $C_1 \times_{{\bf P}^1} C_2$
is irreducible.
Denote by $C$ the non-singular projective model of $C_1 \times_{{\bf P}^1} C_2$,
 which is called  a {\em generalised Howe curve} in \cite{kt23}.
Note that a curve $C$ is just a Howe curve when $g_1= g_2 = 1$ in \cite{how16}. 

\begin{proposition}[{\cite[Proposition 1]{kt23}}] \label{genus}
The genus $g$ of $C$ is equal to $2(g_1 + g_2) + 1 - r$.
\end{proposition}
\begin{theorem}[{\cite[Theorem 1]{kt23}}] \label{hyper}
Assume $g\ge 4$.  Then the curve $C$ is hyperelliptic if and only if $r=g_1+g_2+1$.
\end{theorem}

Originally the Jacobian of a generalised Howe curve
was decomposed over an algebraically closed field in \cite{kt23}.
Now, we should extend it to a field $k$.
The next theorem by Kani and Rosen is a powerful tool for us to solve this problem.
From now on, we denote by $J(C)$ the Jacobian of $C$.
\begin{theorem}[{\cite[Theorem B]{kr89}}] \label{kanirosen}
Let $C$ be a curve over $k$ and $G$ a finite subgroup of
the automorphism group $\operatorname{Aut}(C)$ such that
$G=H_1\cup \cdots \cup H_n$, where the $H_i$'s are subgroups 
of $G$ such that $H_i\cap H_j=\{1_G\}$ for $i\neq j$.
Then we have the isogeny relation
\begin{align*}
J(C)^{n-1}\times J(C/G)^g \sim J(C/H_1)^{h_1}\times \cdots \times J(C/H_n)^{h_n}
\end{align*}
where $g=|G|$ and $h_i=|H_i|$.
\end{theorem}
\begin{theorem}\label{jacghc}
With notation as above, a generalised Howe curve $C$ 
decomposes over $k$ as follows\textup{:}
\begin{align*}
J(C) \sim J(C_1) \times J(C_2) \times J(C_3)
\end{align*}
where $C_3$ is defined by
\begin{align*}
C_3 \colon y_3^2=(x-a_{r+1})\ldots(x-a_{2g_1+2})(x-b_{r+1})\ldots (x-b_{2g_1+2}).
\end{align*}
\end{theorem}
\begin{proof}
Considering two automorphisms of the curve $C$:
\begin{align*}
\sigma&\colon (x, y_1, y_2) \mapsto (x, -y_1, y_2),\\
\tau&\colon (x, y_1, y_2) \mapsto  (x, y_1, -y_2).
\end{align*}
Then, the quotients $C\slash \langle\sigma \rangle$, 
 $C\slash \langle\tau\rangle$ and
 $C\slash \langle\sigma\tau\rangle$ 
are birational to curves $C_2$, $C_1$ and $C_3$ respectively.
Applying Theorem~\ref{kanirosen} to the subgroup of the automorphism group of $C$ generated by $\sigma$ and
$\tau$, we have the isogeny
\begin{align*}
J(C)^2\times J(C/\langle \sigma, \tau \rangle )^4 \sim J(C/\langle \sigma \rangle)^2\times J(C/\langle \tau \rangle)^2 \times J(C/\langle \sigma\tau\rangle)^2,
\end{align*}
which means that the Jacobian of $C$ decomposes as $ J(C) \sim J(C_1) \times J(C_2) \times J(C_3).$
\end{proof}
\begin{corollary}\label{nrp}
If $k=\F_q$ then the number of rational points of $C$ over $\F_q$ 
\begin{align*}
 \#C(\F_q)=\#C_1(\F_q)+\#C_2(\F_q)+\#C_3(\F_q)-2q-2.
\end{align*}
\end{corollary}
\begin{proof} It is well known that $\#C(\F_q)=q+1-t$, where $t$ is the trace of the Frobenius endomorphism
acting on a Tate module of $J(C)$.
Since $J(C) \sim J(C_1) \times J(C_2) \times J(C_3)$,
then the Tate module of $J(C)$ is isomorphic to the direct sum of the Tate modules of $J(C_1)$, $J(C_2)$ and $J(C_3)$.
Hence $t=t_1+t_2+t_3$, where $t_1$, $t_2$ and $t_3$ are the traces of the Frobenius on the Tate modules of $J(C_1)$, $J(C_2)$ and $J(C_3)$ respectively.
The result follows by recalling that $t_i=q+1-\#C_i(\F_q)$ for $1 \le i \le 3$.
\end{proof}

\section{Twisted Legendre elliptic curves over finite fields}
\label{twistedE}
Throughout this section, let 
$\theta \in \F_p\backslash \{0\}$ and $\lambda \in \F_p\backslash \{0, 1\}$, and
a twisted Legendre elliptic curve is defined by 
\begin{align*} 
 E_\lambda^{(\theta)}\colon  y^2 = \theta x(x-1)(x-\lambda).
\end{align*}
Let $p \ge 3$ and $m=(p-1)/2$. We define a polynomial
\begin{align*} 
H_p(t)=\sum_{i=0}^m \binom{m}{i}^2 t^i
\end{align*}
as in \cite[V.4, Theorem 4.1]{s09}.
\begin{theorem}  \label{elserre}
\begin{enumerate}
\item[(i)]
Let $p \ge 17$. 
A curve $E_\lambda^{(\theta)}$ over $\F_p$ attains the Serre bound if and only if
\begin{align*}
(-\theta)^mH_p(\lambda )\equiv - \lfloor 2 \sqrt{p}\rfloor \mod p. 
\end{align*}
\item[(ii)]
A curve $E_\lambda^{(\theta)}$ over $\F_{p^2}$ is maximal if and only if
\begin{align*}
H_p(\lambda) \equiv 0 \mod p.
\end{align*}
Further, if $E_\lambda^{(\theta)}$ over $\F_{p^2}$ is maximal then 
\begin{align*}
p \equiv 3 \mod 4. 
\end{align*}
\item[(iii)]
Let $p \ge 11$. 
 Set $h$ as the integer such that
$h \equiv (-\theta)^m H_p(\lambda)$ mod $p$ and $0 \leqq h <p$.
Then a curve $E_\lambda^{(\theta)}$ over $\F_{p^3}$ attains the Serre bound
if and only if
\begin{align*}
h^3-3p  h=- \lfloor 2p \sqrt{p} \rfloor.
\end{align*}
\end{enumerate}
\end{theorem}
\begin{proof}
We need the coefficient of $x^{p-1}$ in the expression $(\theta x(x-1)(x-\lambda))^m$.
It is the same as the coefficient of $x^m$ in $\theta^m(x-1)^m(x-\lambda)^m$,
which is equal to
\begin{align*}
\theta^m\sum_{i=0}^m \binom{m}{i}(-\lambda)^i\binom{m}{m-i}(-1)^{m-i}
=(-\theta)^m H_p(\lambda).
\end{align*}
\begin{enumerate}
\item[(i)]
Since the coefficient of $x^{p-1}$ in $(\theta x(x-1)(x-\lambda))^m$
is $(-\theta)^m H_p(\lambda)$,
it follows from \cite[Theorem 2]{k15} that a curve 
$E_\lambda^{(\theta)}$ over $\F_p$ attains the Serre bound if and only if
\begin{align*}
(-\theta)^m H_p(\lambda) \equiv -\lfloor 2 \sqrt{p} \rfloor \mod p.
\end{align*}
\item[(ii)]
Similarly to the proof of \cite[V.4, Theorem 4.1]{s09},
we have that a curve $E_\lambda^{(\theta)}$ over $\F_{p^2}$ is maximal if and only if
\begin{align*}
(-\theta)^m H_p(\lambda) \equiv 0 \mod p.
\end{align*}
Since $\theta \neq 0$, the first claim follows.
On the other hand, Proposition 3.2 (1) of \cite{at02} says that
the set $\{\lambda \in \F_p | H_p(\lambda)=0\}$ is empty if and only if $p \equiv 1$ mod $4$.
Hence, only when $p \equiv 3$ mod $4$, 
$E_\lambda^{(\theta)}$ can be maximal over $\F_{p^2}$.
\item[(iii)] From \cite[Theorem 4]{k15}, we obtain it immediately.
\end{enumerate}
\end{proof}

\begin{lemma} \label{mod4}
 The number of rational points of $E_{\lambda}^{(\theta)}$ over $\F_q$
\begin{align*}
\#E_{\lambda}^{(\theta)}(\F_q) \equiv 0 \mod 4.
\end{align*}
\end{lemma}
\begin{proof}
When $\theta=1$, 
$\#E_{\lambda}^{(1)}(\F_p) \equiv 0$ mod $4$ from \cite[Section 1]{katz10}.
When $\theta \notin {\F_p^*}^2$,
$\#E_{\lambda}^{(1)}(\F_p)+ \#E_{\lambda}^{(\theta)}(\F_p) =2p+2$ from \cite[Section 2]{at02}.
Hence we have 
$\#E_{\lambda}^{(\theta)}(\F_p) \equiv 0$ mod $4$.

Next, set $n_i=\#E_\lambda^{(\theta)}(\F_{p^i})$, 
$a_1=p+1-n_1$, $a_2=a_1^2-2p$ and $a_j=a_1a_{j-1}-pa_{j-2}$ for $j \ge 3$.
Applying the theory of Zeta function, we have $n_j=p^j+1-a_j$.
Assume $n_1 \equiv 0$ mod $4$.  Then, we have $a_j \equiv p^j +1$ mod $4$ by induction.
Therefore $n_j \equiv 0$ mod $4$. 
\end{proof}

\section{Twisted Generalised Howe curves of genus five} \label{genus5}
For our purpose to discover curves attaining the Serre bound,
we deal with the next type of twisted generalised Howe curves over a field $k$. 
Throughout this section, we set $C$ as the fibre product $C_1 \times_{{\bf P}^1} C_2$,
where $C_1$ and $C_2$ are curves of genus $2$ defined as follows:
\begin{align*}
C_1 &\colon y_1^2 = \alpha_1(x-a_1)(x-a_2)(x-a_3)(x-a_4)(x-a_5)(x-a_6),\\
C_2 &\colon y_2^2 = \alpha_2(x-a_1)(x-a_2)(x-a_3)(x-a_4)(x-b_5)(x-b_6)
\end{align*}
with $\alpha_1, \alpha_2 \in k\backslash\{0\} $, $a_i, b_i \in k$, 
where $a_1, \ldots, a_6$ and $b_5, b_6$ are all different.
Clearly, if $\alpha_1, \alpha_2 \in {k^*}^2$
then $C$ is birational to a generalised Howe curve.
From Proposition \ref{genus} and Theorem \ref{hyper},
it is a non-hyperelliptic curve of genus $5$.
Besides, by extending Theorem \ref{jacghc} to a twisted generalised Howe curve,  
we have the Jacobian decomposition $ J(C) \sim J(C_1) \times J(C_2) \times J(C_3)$, 
where $C_3$ is defined as follows:
\begin{align*}
C_3\colon y_3^2 = \alpha_1\alpha_2(x-a_5)(x-a_6)(x-b_5)(x-b_6).
\end{align*}

To decompose Jacobians of curves $C_1$ and $C_2$,
we extend Theorem 2 (b) of \cite{ikt24} from a finite field $\F_q$ to a field $k$ as the next theorem.
Because the proofs are similar, we omit it here. 
\begin{theorem} 
 \label{JacHyper} 
Let a curve of genus $2$ be defined by
\begin{align*}
D\colon y^2=\alpha(x-a_1)(x-a_2)(x-a_3)(x-a_4)(x-a_5)(x-a_6)
\end{align*}
with $\alpha \in k\backslash\{0\}$, $a_i \in k$,
$a_i \neq a_j$ when $i \neq j$ and
$(a_2-a_4)(a_1-a_6)(a_3-a_5)=(a_2-a_6)(a_1-a_5)(a_3-a_4).$
Set $\lambda=\dfrac{(a_1-a_3)(a_2-a_4)}{(a_2-a_3)(a_1-a_4)}$,
$\mu=\dfrac{(a_1-a_3)(a_2-a_5)}{(a_2-a_3)(a_1-a_5)}$
and $\theta=\alpha\cdot(a_2-a_3)(a_1-a_4)(a_1-a_5)(a_1-a_6).$
Assume that there exists a square root of $\lambda(\lambda-\mu)$ in $k^*$.

Then the Jacobian of the curve $D$ decomposes over $k$ as 
\begin{align*}
J(D) \sim E_+ \times E_-, 
\end{align*}
where we have the following defining equations\textup{:}
\begin{align*}
 s^2=\dfrac{\theta(1-\mu)}{1-\lambda}t(t-1)
\Big(t-\dfrac{(1-\lambda) \big(\mu-2\lambda\pm2(\lambda^2-\lambda\mu)^{1/2}\big)}{ \mu-1}\Big)\\
\end{align*}
for $E_+$ and $E_-$ respectively.
\end{theorem}
Afterward in this section, we assume that
\begin{align*}
(a_2-a_4)(a_1-a_6)(a_3-a_5)&=(a_2-a_6)(a_1-a_5)(a_3-a_4),\\ 
(a_2-a_4)(a_1-b_6)(a_3-b_5)&=(a_2-b_6)(a_1-b_5)(a_3-a_4),
\end{align*}
and that both $(a_1-a_2)(a_2-a_4)(a_4-a_5)(a_5-a_1)$ and $(a_1-a_2) (a_2-a_4)(a_4-b_5)(b_5-a_1)$ are square roots in $k^*$.

Also we set
\begin{align*}
a=\dfrac{(a_1-a_3)(a_2-a_4)}{(a_2-a_3)(a_1-a_4)},\quad
b=\dfrac{(a_1-a_3)(a_2-a_5)}{(a_2-a_3)(a_1-a_5)}, \quad
c=\dfrac{(a_1-a_3)(a_2-b_5)}{(a_2-a_3)(a_1-b_5)},
\end{align*}
$\beta_1=\alpha_1(a_2-a_3)(a_1-a_4)(a_1-a_5)(a_1-a_6)$
and $\beta_2=\alpha_2(a_2-a_3)(a_1-a_4)(a_1-b_5)(a_1-b_6)$.

Next let
\begin{align*}
\theta_1&=\theta_2=\dfrac{\beta_1 (1-b)}{1-a},\qquad 
\lambda_1, \ \lambda_2 =\dfrac{(1-a)\big(b-2a\pm 2(a^2-ab)^{1/2}\big)}{b-1},\\
\theta_3&=\theta_4=\dfrac{\beta_2 (1-c)}{1-a},\qquad
\lambda_3, \ \lambda_4 =\dfrac{(1-a)\big(c-2a\pm 2(a^2-ac)^{1/2}\big)}{c-1}, \\
\theta_5&=\alpha_1 \alpha_2(a_5-b_6)(a_6-b_5),\qquad
\lambda_5=\dfrac{(a_5-b_5)(a_6-b_6)}{(a_5-b_6)(a_6-b_5)}.
\end{align*}
\begin{theorem} \label{g5jac} 
With the assumptions and notation as above,
the Jacobian of the curve $C$ has the following isogeney relation over $k$\textup{:}
\begin{align*}
J(C) \sim E_1 \times E_2 \times E_3 \times E_4 \times E_5 
\end{align*}
with the five elliptic curves defined by
\begin{align*}
E_i\colon s^2=\theta_i t(t-1)(t-\lambda_i) \qquad \text{for}\  1 \le i \le 5. 
\end{align*}

In particular, if $k=\F_q$ then the number of rational points of $C$ over $\F_q$
\begin{align*}
\#C(\F_q) =\sum_{i=1}^5  \#E_i(\F_q)-4q-4.
\end{align*}
\end{theorem}

\begin{proof}
Because $(a_1-a_2)(a_2-a_4)(a_4-a_5)(a_5-a_1)$
and $(a_1-a_2) (a_2-a_5)(a_5-b_5)(b_5-a_1)$ are square roots in $k^*$,
so are the elements
$a(a-b)=\dfrac{(a_1-a_3)^2(a_2-a_4)(a_2-a_1)(a_4-a_5)}{(a_2-a_3)^2(a_1-a_4)^2(a_1-a_5)}$ and 
$a(a-c)=\dfrac{(a_1-a_3)^2(a_2-a_4)(a_2-a_1)(a_4-b_5)}{(a_2-a_3)^2(a_1-a_4)^2(a_1-b_5)}$.
Applying Theorem \ref{JacHyper}  to curves $C_1$ and $C_2$,
we obtain their Jacobian decompositions as 
$J(C_1) \sim E_1 \times E_2$ and $J(C_2) \sim E_3 \times E_4$ respectively.
Since $C_3$ is birational to $E_5$, we can prove it.

Next, let $k=\F_q$. It is well known that $\#C(\F_q)=q+1-t$, where $t$ is the trace of the Frobenius endomorphism
acting on a Tate module of $J(C)$. Since $J(C) \sim E_1 \times \cdots \times E_5$,
then the Tate module of $J(C)$ is isomorphic to the direct sum of the Tate modules of $E_1, \ldots, E_5$.
Hence $t=t_1+\cdots+t_5$, where $t_1, \ldots, t_5$ are the traces of
the Frobenius on the Tate modules of $E_1,\ldots,E_5$ respectively.
The result follows by recalling that $t_i=q+1-\#E_i(\F_q)$ for $1 \le i \le 5$.
\end{proof}

\begin{theorem}\label{g5serre} 
Suppose further that $\alpha_1, \alpha_2$, $a_1, \ldots, a_6, b_5, b_6 \in \F_p$
and that both $(a_1-a_2)(a_2-a_4)(a_4-a_5)(a_5-a_1)$
and $(a_1-a_2) (a_2-a_5)(a_5-b_5)(b_5-a_1)$ 
are square roots in $\F_p^*$.
\begin{enumerate}
\item[(i)]\label{p1}
Let $p \ge 17$.  The curve $C$ over $\F_p$ attains the Serre bound
if and only if
\begin{align*}
(-\theta_i)^m H_p(\lambda_i) &\equiv - \lfloor 2 \sqrt{p}\rfloor \mod p 
\qquad \text{for}\ 1\le i \le 5.
\end{align*}
\item[(ii)]
The curve $C$ over $\F_{p^2}$ is maximal if and only if
\begin{align*}
H_p(\lambda_i) \equiv 0 \mod p
\qquad \text{for}\ 1\le i \le 5.
\end{align*}
Further, if $C$ over $\F_{p^2}$ is maximal then 
\begin{align*}
p \equiv 3 \mod 4. 
\end{align*}
\item[(iii)]
Let $p \ge 11$.  Set $h_i$ as the integer such that
$h_i \equiv (-\theta_i)^m H_p(\lambda_i)$ mod $p$ and $0 \leqq h_i <p$.
The curve $C$ over $\F_{p^3}$ attains the Serre bound
if and only if
\begin{align*}
h_i^3-3p h_i=- \lfloor 2p \sqrt{p} \rfloor
\qquad \text{for} \ 1\le i \le 5.
\end{align*}
\item[(iv)]
The number of rational points of $C$ over $\F_q$
\begin{align*}
 \#C(\F_q) \equiv 0 \mod 4.
\end{align*}
\end{enumerate}
\end{theorem}
\begin{proof}
From Theorem \ref{g5jac}, we have the isogeny relation 
$J(C) \sim E_1 \times \cdots \times E_5$ with 
$E_i\colon s^2=\theta_i t(t-1)(t-\lambda_i)$ 
and the number of rational points
$\#C(\F_q) =\sum_{i=1}^5  \#E_i(\F_q)-4q-4$.
Hence, a curve $C$ over $\F_q$ attains the Serre bound if and only if 
$E_i$ over $\F_q$ attains the Serre bound for all $1\le i \le5$. 
Combine it with Theorem \ref{elserre} (i), (ii) and (iii),
we are able to prove (i), (ii) and (iii) respectively.

Recalling Lemma \ref{mod4}, we obtain (iv) immediately.
\end{proof}

Table  \ref{tableg5p1}
lists explicit values $(p, \alpha_1, \alpha_2, a_1, \ldots, a_6,b_5,b_6)$ satisfying
if and only if conditions of Theorem \ref{g5serre} (i).
They are new curves of genus $5$ attaining the Serre bound over $\F_{p}$. 
Example~\ref{g5p499} explains the case of $p=499$.
The other cases in the table are similar to it.
\begin{table}[h] 
\caption{Curves of genus $5$ attaining the Serre bound over $\F_p$}
\label{tableg5p1}
\begin{tabular}{ccccccccccc}
\hline
$p$&$\alpha_1$&$\alpha_2$&$a_1$&$a_2$&$a_3$&$a_4$&$a_5$&$a_6$&$b_5$&$b_6$\\
\hline
499&47&436&2&1&10&55&92&84&36&275\\
599&501&399&3&2&24&276&97&32&94&55\\
1187&692&739&5&3&29&11&58&726&125&490\\
\hline
\end{tabular}
\end{table}

\begin{example}\label{g5p499} 
Two curves of genus $2$ have the following defining equations: 
\begin{align*}
C_1&\colon y_1^2=\ \, 47(x-2)(x-1)(x-10)(x-55)(x-92)(x-84),\\
C_2&\colon y_2^2=436(x-2)(x-1)(x-10)(x-55)(x-36)(x-275).
\end{align*}
The curve $C$ of genus $5$, which is defined by the fibre product of $C_1 \times_{{\bf P}^1} C_2$, 
attains the Serre bound over $\F_{499}$.
Note that since $47, 436 \in {\F_{499}^*}^2$ two curves
$C_1$ and $C_2$ are birational to 
$y_1^2=(x-2)(x-1)(x-10)(x-55)(x-92)(x-84)$ and
$y_2^2=(x-2)(x-1)(x-10)(x-55)(x-36)(x-275)$ respectively.

The Jacobian $J(C) \sim J(C_1) \times J(C_2) \times J(C_3)$ with
$C_3\colon y_3^2=47\cdot436(x-92)(x-84)(x-36)(x-275)$, which is birational to $y_3^2=(x-92)(x-84)(x-36)(x-275)$.
Applying Theorem \ref{g5jac} to the curve $C$, 
its Jacobian completely decomposes as
$J(C) \sim E_1 \times \cdots \times E_5$
where the five elliptic curves are defined by
$E_1\colon s^2=31t(t-1)(t-438)$, 
$E_2\colon s^2=31t(t-1)(t-198)$, 
$E_3\colon s^2=95t(t-1)(t-62)$, 
$E_4\colon s^2=95t(t-1)(t-302)$, 
$E_5\colon s^2=342t(t-1)(t-198)$. 
Here $31,\, 342 \in {\F_{499}^*}^2$, $95 \notin {\F_{499}^*}^2$.
\end{example}

Table  \ref{tableg5p2}
lists explicit values $(p, \alpha_1, \alpha_2, a_1, \ldots, a_6,b_5,b_6)$ satisfying
if and only if conditions of Theorem \ref{g5serre} (ii).
They are new maximal curves of genus $5$ over $\F_{p^2}$.
Even we set $\alpha_1=\alpha_2=1$ in this table,
they are still maximal curves over $\F_{p^2}$, 
because $\alpha_1, \alpha_2 \in \F_p$ are square roots in $\F_{p^2}^*$.
We explain the case of $p=11$ in Example \ref{g5p11},
where the same approach can be used to the other cases in the table.
\begin{table}[h] 
\caption{Maximal curves of genus $5$ over $\F_{p^2}$}
\label{tableg5p2}
\begin{tabular}{ccccccccccc}
\hline
$p$&$\alpha_1$&$\alpha_2$&$a_1$&$a_2$&$a_3$&$a_4$&$a_5$&$a_6$&$b_5$&$b_6$\\
\hline
11&4&6&5&3&10&7&6&8&9&2\\
23&16&8&5&3&9&7&11&13&22&1\\
31&10&7&6&7&11&15&14&10&2&19\\
43&38&24&20&19&15&40&42&22&8&29\\
47&31&26&6&13&7&4&18&2&45&8\\
59&5&51&4&8&2&33&54&17&21&40\\
71&36&18&4&9&3&23&41&45&61&69\\
79&11&9&11&36&14&66&49&35&27&72\\
83&4&37&2&3&1&48&54&80&7&19\\
103&17&25&2&3&1&58&61&75&85&14\\
107&83&104&7&5&2&29&56&16&101&47\\
127&68&87&6&5&3&28&38&39&99&48\\
131&59&55&6&5&17&2&79&61&34&89\\
139&107&118&6&7&4&68&35&88&50&93\\
151&45&62&11&9&7&150&70&38&37&114\\
167&72&166&10&3&5&147&142&13&38&144\\
179&167&128&12&9&15&97&175&11&52&139\\
191&115&150&13&90&76&1&46&128&88&79\\
199&32&125&113&20&103&194&4&33&70&59\\
\hline
\end{tabular}
\end{table}

\begin{example}\label{g5p11}
The curve $C$ of genus $5$, which is defined by the fibre product $C_1 \times_{{\bf P}^1} C_2$ with
\begin{align*}
C_1& \colon y_1^2=4(x-5)(x-3)(x-10)(x-7)(x-6)(x-8),\\
C_2& \colon y_2^2=6(x-5)(x-3)(x-10)(x-7)(x-9)(x-2)
\end{align*}
is maximal over $\F_{11^2}$.
Here $4 \in {\F_{11}^*}^2$, $6 \notin {\F_{11}^*}^2$.
Hence in particular a curve $C_1$ is birational to $y_1^2=(x-5)(x-3)(x-10)(x-7)(x-6)(x-8)$.

The Jacobian $ J(C) \sim J(C_1) \times J(C_2) \times J(C_3)$ with
$C_3 \colon y_3^2= 4\cdot6 (x-6)(x-8)(x-9)(x-2).$
Furthermore, Theorem \ref{g5jac} gives us the complete decomposition of the Jacobian:
$ J(C) \sim E_1 \times \cdots \times E_5$
with
$E_1\colon s^2=8t(t-1)(t-6)$, 
$E_2\colon s^2=8t(t-1)(t-2)$, 
$E_3\colon s^2=8t(t-1)(t-2)$,
$E_4\colon s^2=8t(t-1)(t-10)$, 
$E_5\colon s^2=3t(t-1)(t-10)$. 
Note $8 \notin {\F_{11}^*}^2$, $3 \in {\F_{11}^*}^2$.
\end{example}

Table  \ref{tableg5p3}
lists explicit values $(p, \alpha_1, \alpha_2, a_1, \ldots, a_6,b_5,b_6)$ satisfying
if and only if conditions of Theorem \ref{g5serre} (iii).
They are new curves of genus $5$ attaining the Serre bound over $\F_{p^3}$. 
We explain the case of $p=37$ in Example~\ref{g5p37},
where the other cases are similar to it.
\begin{table}[h] 
\caption{Curves of genus $5$ attaining the Serre bound over $\F_{p^3}$}
\label{tableg5p3}
\begin{tabular}{ccccccccccc}
\hline
$p$&$\alpha_1$&$\alpha_2$&$a_1$&$a_2$&$a_3$&$a_4$&$a_5$&$a_6$&$b_5$&$b_6$\\
\hline
37&17&6&0&1&3&31&34&13&29&30\\
97&81&91&2&3&1&85&11&69&76&8\\
193&79&22&1&2&4&177&127&66&52&156\\
\hline
\end{tabular}
\end{table}

\begin{example} \label{g5p37}
The curve $C$ of genus $5$,
which is defined by the fibre product $C_1 \times_{{\bf P}^1} C_2$ with
\begin{align*}
C_1& \colon y_1^2=17x(x-1)(x-3)(x-31)(x-34)(x-13),\\
C_2& \colon y_2^2=6x(x-1)(x-3)(x-31)(x-29)(x-30)
\end{align*}
attains the Serre bound over $\F_{37^3}$. Its Jacobian
$ J(C) \sim J(C_1) \times J(C_2) \times J(C_3)$ with
$C_3 \colon y_3^2=17 \cdot 6(x-34)(x-13)(x-29)(x-30)$.
Here $17, 6 \notin {\F_{37}^*}^2$, $17\cdot 6\in {\F_{37}^*}^2$.
Moreover, the Jacobian $J(C) \sim E_1 \times \cdots \times E_5$
with
$E_1\colon s^2=26t(t-1)(t-26)$,
$E_2\colon s^2=26t(t-1)(t-4)$, 
$E_3\colon s^2=4t(t-1)(t-12)$,
$E_4\colon s^2=4t(t-1)(t-34)$, 
$E_5\colon s^2=30t(t-1)(t-10)$.
Here $26, 4, 30 \in {\F_{37}^*}^2$.
\end{example}


\begin{thebibliography}{99}
\bibitem{at02}
R.~Auer, J.~Top,
Legendre elliptic curves over finite field,
J.\ Number Theory \textbf{95} (2002), 303--312. 

\bibitem{magma}
W.~Bosma, J.~Cannon, C.~Playoust, 
The Magma algebra system. I. The user language, 
J.\ Symbolic Comput.\ \textbf{24} (1997), 235--265.

\bibitem{cs20}
C.~Costello, B.~Smith,
The supersingular isogeny problem in genus 2 and beyond,
PQCrypto 2020, LNCS \textbf{12100} (2020), 151--168.

\bibitem{ggs10}
A.~Garcia, G.~G\"uneri, H.~Stichtenoth,
A generalization of the Giulietti--Korchm\'aros maximal curve,
Adv.\ Geom.\ \textbf{10}(3) (2010), 427--434.

\bibitem{gt}
A.~Garcia, S.~Tafazolian,
Certain maximal curves and Cartier operators,
Acta Arith.\ \textbf{135}(39) (2008), 199--218.

\bibitem{ghlr}
G.~van der Geer, E.~Howe, K.~Lauter, C.~Ritzenthaler,
Table of curves with many points,
\url{http://www.manypoints.org}.

\bibitem{gmz15}
M.~Giulietti, M.~Montanucci, G.~Zini,
On maximal curves that are not quotients of the Hermitian curve,
Finite Fields Appl.\ \textbf{41} (2016), 72--88.

\bibitem{goppa}
V.~D.~Goppa, Codes on algebraic curves,
Dokl. Akad. nauk SSSR \textbf{259}(6) (1981), 1289--1290.

\bibitem{how16}
E.~W.~Howe, Quickly constructing curves of genus 4 with many points,
Comtemp.\ Math.\ \textbf{663} (2016), 149--173.

\bibitem{how17}
E.~ W.~Howe,
Curves of medium genus with many points,
Finite Fields Appl.\ \textbf{47} (2017), 145--160. 

\bibitem{ikt24}
A.~Iezzi, M.~Q.~Kawakita, M.~Timpanella,
New sextics of genus 6 and 10 attaining the Serre bound,
Adv.\ Geometry \textbf{24}(1) (2024), 99--109.

\bibitem{kr89}
E.~Kani, M.~Rosen,
Idempotent relations and factors of Jacobians,
Math.\ Ann.\ \textbf{284}(2) (1989), 307--327.

\bibitem{kt20}
T.~Katsura, K.~Takashima,
Counting Richelot isogenies between superspecial abelian surfaces,
ANTS 2020, The open book series \textbf{4} (2020), 283--300.

\bibitem{kt23}
T.~Katsura, K.~Takashima,
Decomposed Richelot isogenies of Jacobian varieties of hyperelliptic curves and generalized Howe curves,
to appear in Commentarii Mathematici Univ.\ St.\ Pauli,
arXiv:2108.06936.

\bibitem{katz10}
N.~M.~Katz,
2, 3, 5, Legendre: $\pm$trace ratio in families of elliptic curves,
Experimental Math.\ \textbf{19}(3) (2010), 267--277.

\bibitem{k15}
M.~Q.~Kawakita, 
Wiman's and Edge's sextic attaining Serre's bound II, 
Contemp.\ Math.\ \textbf{637} (2015), 191--203. 

\bibitem{khh20}
M.~Kudo, S.~Harashita, E.~Howe,
Algorithms to enumerate superspecial {Howe} curves of genus four,
ANTS 2020, The open book series \textbf{4} (2020), 301--316.

\bibitem{khs20}
M.~Kudo, S.~Harashita, H.~Senda,
The existence of supersingular curves of genus 4 in arbitrary characteristic,
Res.\ Number Theory \textbf{6},44 (2020).

\bibitem{s83}
J-P.~Serre,
Sur le nombre des points rationnels d'une courbe alg\'ebrique sur un corps fini,
C.\ R.\ Acad.\ Sci.\ Paris S\'er.\ I Math.\ \textbf{296}(9) (1983), 397--402.

\bibitem{s09}
J.~H.~Silverman,
The arithmetic of elliptic curves 2nd Ed., GTM \textbf{106},
Springer 2009.

\bibitem{sti}
H.~Stichtenoth, Algebraic function fields and codes 2nd Ed., GTM \textbf{254},
Springer 2009.
\end{thebibliography}
\end{document}